\documentclass[reqno,12pt]{amsart}
\usepackage{amsfonts}
\usepackage{bbm}
\usepackage{}
\numberwithin{equation}{section}
\usepackage{color}
\usepackage{rotating}
\usepackage{indentfirst}
\usepackage{color}
\usepackage{amssymb}
\usepackage{mathrsfs}
\usepackage{xy}
\usepackage{rotating}
\usepackage{tikz}

\xyoption{all}
 \def\Hom{\mbox{\rm Hom}} \def\dim{\mbox{\rm dim}\,} 
\def\lr#1{\langle #1\rangle} \def\fin{\hfill$\square$}   
    
\def\End{\mbox{\rm End}\,}\def\cone{\mbox{\rm cone}}\def\cocone{\mbox{\rm cocone}}
\def\ind{\mbox{\rm ind}\,}
\def\Aut{\mbox{\rm Aut}\,} 
 \def\Hom{\mbox{\rm Hom}\,}  \def\Im{\mbox{\rm Im}\,} \def\Im{\mbox{\rm Im}\,} \def\dim{\mbox{\rm dim}}\def\rad{\mbox{\rm rad}}

\theoremstyle{plain}
\newtheorem{theorem}{\bf Theorem}[section]
\newtheorem{lemma}[theorem]{\bf Lemma}
\newtheorem{corollary}[theorem]{\bf Corollary}
\newtheorem{proposition}[theorem]{\bf Proposition}

\theoremstyle{definition}
\newtheorem{definition}[theorem]{\bf Definition}
\newtheorem{remark}[theorem]{\bf Remark}
\newtheorem{example}[theorem]{\bf Example}

\newtheorem{condition}[theorem]{\bf Condition}

\newcommand{\bt}{\begin{theorem}}
\newcommand{\et}{\end{theorem}}
\newcommand{\bl}{\begin{lemma}}
\newcommand{\el}{\end{lemma}}
\newcommand{\bd}{\begin{definition}}
\newcommand{\ed}{\end{definition}}
\newcommand{\bc}{\begin{corollary}}
\newcommand{\ec}{\end{corollary}}
\newcommand{\bp}{\begin{proof}}
\newcommand{\ep}{\end{proof}}
\newcommand{\bx}{\begin{example}}
\newcommand{\ex}{\end{example}}
\newcommand{\br}{\begin{remark}}
\newcommand{\er}{\end{remark}}
\newcommand{\be}{\begin{equation}}
\newcommand{\ee}{\end{equation}}
\newcommand{\ba}{\begin{align}}
\newcommand{\ea}{\end{align}}
\newcommand{\bn}{\begin{enumerate}}
\newcommand{\en}{\end{enumerate}}
\newcommand{\bcs}{\begin{cases}}
\newcommand{\ecs}{\end{cases}}
\newcommand{\RNum}[1]{\uppercase\expandafter{\romannumeral #1\relax}}

\makeatletter
\renewcommand{\section}{\@startsection{section}{1}{0mm}
  {-\baselineskip}{0.5\baselineskip}{\bf\leftline}}
\makeatother

\begin{document}
\title[Hall algebras of extriangulated categories]{Hall algebras of extriangulated categories}
\author[L. Wang, J. Wei, H. Zhang]{Li Wang, Jiaqun Wei,  Haicheng Zhang}
\address{Institute of Mathematics, School of Mathematical Sciences, Nanjing Normal University,
 Nanjing 210023, P. R. China.\endgraf}
\email{wl04221995@163.com (Wang); weijiaqun@njnu.edu.cn (Wei); zhanghc@njnu.edu.cn (Zhang).}


\begin{abstract}
Recently, Nakaoka and Palu introduced a notion of extriangulated categories. This is a
unification of exact categories and triangulated categories. In this paper,
we generalize the definitions of Hall algebras of exact categories and triangulated categories to extriangulated categories.
\end{abstract}

\maketitle
\section{Introduction}
Hall algebras first appeared in works of Steinitz \cite{St} and Hall \cite{Hall} on commutative finite $p$-groups.
Later, they reappeared in the work of Ringel \cite{Ringel1} on quantum groups. Ringel introduced the notion of the Hall algebra of an abelian category with finite ${\rm Hom}$- and ${\rm Ext}^1$-spaces. By definition, it is a vector space with the basis parameterized by the isomorphism classes of
objects in the category, and the structure constants of the multiplication count in a natural
way the first extensions with a fixed middle term. The associativity formula of the Hall algebra of abelian categories is easily obtained by counting the filtration of an object. According to \cite{Ringel1} and \cite{Green}, the Hall algebra of a finite dimensional hereditary algebra over a finite filed provides a realization of the half part of the corresponding quantum group.

Hubery \cite{Hu} proved that the definition of the Hall algebra of abelian categories also applies to exact categories. He proved the associativity formulas of Hall algebras via the push-out and pull-back properties of exact categories. In order to give a realization of the entire quantum group via the Hall algebra approach, one tried to define the Hall algebra of triangulated categories. To\"{e}n \cite{To} gave a construction of
what he called derived Hall algebras for DG-enhanced triangulated categories satisfying
certain finiteness conditions. Later, Xiao and Xu \cite{XF} showed that To\"{e}n's definition also applies to any triangulated categories satisfying
certain finiteness conditions. They proved the associativity formulas of the derived Hall algebra by using only the properties and axioms of triangulated categories, such as the long exact sequence theorem and the octahedron axiom.

Recently, Nakaoka and Palu \cite{Na} has introduced a notion of extriangulated categories. This is a
unification of exact categories and triangulated categories. The positive and negative higher extensions in extriangulated categories have been defined in \cite{Go}.
A natural question has been asked by many people: Can one define
the Hall algebra of an extriangulated category? We refer to
\cite[Section 7]{FG} for a discussion of this question and the announcement
of a positive answer.

In this paper, we give a positive answer
based on the methods of \cite{XF}. We prove that Toen's formulas still
hold in an extriangulated category.
But this is not a trivial generalization. Compared with triangulated categories, the extriangulated categories have no shift functor $[1]$, i.e., for an object $Z$, the object $Z[1]$ does not exist. So, for an object $L$, we do not have the morphism $L\rightarrow Z[1]$ or its decomposition (cf. \cite[Lemma 2.2]{XF}).
In order to overcome this trouble, for an $\mathbb{E}$-triangle $X\stackrel{f}{\longrightarrow}L\stackrel{g}{\longrightarrow}Y\stackrel{\delta}\dashrightarrow$, we consider the decompositions of both $f$ and $g$. Meanwhile, instead of the group action of $\Aut (X,Y)$, we consider both the group action of $\Aut (X,L)$ and that of $\Aut (L,Y)$. Then by combining the two equations respectively provided by the decompositions of $f$ and $g$, we obtain To\"{e}n's formulas in an extriangulated category.

The paper is organized as follows: we summarize some basic definitions and properties of an extriangulated category in Section 2. In Section 3, we prove a key formula (see Theorem \ref{main2}) for a right locally homologically finite extriangulated category. Section 4 is devoted to defining the Hall algebra of extriangulated categories using the obtained formulas as structure constants. In Section 5, we prove To\"{e}n's formulas for a left locally homologically finite extriangulated category satisfying certain conditions. Using To\"{e}n's formulas as structure constants, we also define a Hall algebra for extriangulated categories in Section 6. Finally, for a locally homologically finite extriangulated category, we compare these two Hall algebras.

Throughout this paper, let $k$ be a finite field with $q$ elements, and $\mathscr{C}$ be an essentially small Krull-Schmidt additive $k$-linear category.
For an object $X\in\mathscr{C}$, we denote by $\End X$ and $\Aut X$ the endomorphism ring and the automorphism group of $X$; denote by $1=1_X$ and $[X]$ the identity morphism and the isomorphism class of $X$, respectively;
For a finite set $S$, we denote by $|S|$ its cardinality.

\section{Preliminaries}
Let us recall some notions and properties concerning extriangulated categories from \cite{Na}.

Let $\mathscr{C}$ be an additive category and let $\mathbb{E}$: $\mathscr{C}^{op}\times\mathscr{C}\rightarrow Ab$ be a biadditive functor. For any pair of objects $A$, $C\in\mathscr{C}$, an element $\delta\in \mathbb{E}(C,A)$ is called an {\em $\mathbb{E}$-extension}. The zero element $0\in\mathbb{E}(C,A)$ is called the {\em split $\mathbb{E}$-extension}.
 For any morphism $a\in \mathscr{C}(A,A')$ and $c\in {\mathscr{C}}(C',C)$, we have
$\mathbb{E}(C,a)(\delta)\in\mathbb{E}(C,A')$ and $\mathbb{E}(c,A)(\delta)\in\mathbb{E}(C',A).$
We simply denote them by $a_{\ast}\delta$ and $c^{\ast}\delta$,~respectively.  A morphism $(a,c)$: $\delta\rightarrow\delta'$ of $\mathbb{E}$-extensions is a pair of morphisms $a\in \mathscr{C}(A,A')$ and $c\in {\mathscr{C}}(C,C')$ satisfying the equality $a_{\ast}\delta=c^{\ast}\delta'$.
By Yoneda's lemma, any $\mathbb{E}$-extension $\delta\in \mathbb{E}(C,A)$ induces natural transformations
$$\delta_{\sharp}: \mathscr{C}(-,C)\rightarrow\mathbb{E}(-,A)~~\text{and}~~\delta^{\sharp}: \mathscr{C}(A,-)\rightarrow\mathbb{E}(C,-).$$ For any $X\in\mathscr{C}$, these
$(\delta_{\sharp})_X$ and $(\delta^{\sharp})_X$ are defined by
$(\delta_{\sharp})_X:\mathscr{C}(X,C)\rightarrow\mathbb{E}(X,A), f\mapsto f^\ast\delta$
 and $(\delta^{\sharp})_X:\mathscr{C}(A,X)\rightarrow\mathbb{E}(C,X), g\mapsto g_\ast\delta.$

Two sequences of morphisms $A\stackrel{x}{\longrightarrow}B\stackrel{y}{\longrightarrow}C$ and $A\stackrel{x'}{\longrightarrow}B'\stackrel{y'}{\longrightarrow}C$ in $\mathscr{C}$ are said to be {\em equivalent} if there exists an isomorphism $b\in \mathscr{C}(B,B')$ such that the following diagram
$$\xymatrix{
  A \ar@{=}[d] \ar[r]^-{x} & B\ar[d]_{b}^-{\simeq} \ar[r]^-{y} & C\ar@{=}[d] \\
  A \ar[r]^-{x'} &B' \ar[r]^-{y'} &C  }$$ is commutative.
We denote the equivalence class of $A\stackrel{x}{\longrightarrow}B\stackrel{y}{\longrightarrow}C$ by $[A\stackrel{x}{\longrightarrow}B\stackrel{y}{\longrightarrow}C]$. In addition, for any $A,C\in\mathscr{C}$, we denote as
$$0=[A\stackrel{1\choose0}{\longrightarrow}A\oplus C\stackrel{(0~1)}{\longrightarrow}C].$$
For any two classes $[A\stackrel{x}{\longrightarrow}B\stackrel{y}{\longrightarrow}C]$ and $[A'\stackrel{x'}{\longrightarrow}B'\stackrel{y'}{\longrightarrow}C']$, we denote as
$$[A\stackrel{x}{\longrightarrow}B\stackrel{y}{\longrightarrow}C]\oplus[A'\stackrel{x'}{\longrightarrow}B'\stackrel{y'}{\longrightarrow}C']=
[A\oplus A'\stackrel{x\oplus x'}{\longrightarrow}B\oplus B'\stackrel{y\oplus y'}{\longrightarrow}C\oplus C'].$$

\begin{definition}
Let $\mathfrak{s}$ be a correspondence which associates an equivalence class $\mathfrak{s}(\delta)=[A\stackrel{x}{\longrightarrow}B\stackrel{y}{\longrightarrow}C]$ to any $\mathbb{E}$-extension $\delta\in\mathbb{E}(C,A)$. This $\mathfrak{s}$ is called a {\em realization} of $\mathbb{E}$ if for any morphism $(a,c):\delta\rightarrow\delta'$ with $\mathfrak{s}(\delta)=[\Delta_{1}]$ and $\mathfrak{s}(\delta')=[\Delta_{2}]$, there is a commutative diagram as follows:
$$\xymatrix{
\Delta_{1}\ar[d] & A \ar[d]_-{a} \ar[r]^-{x} & B  \ar[r]^{y}\ar[d]_-{b} & C \ar[d]_-{c}    \\
 \Delta_{2}&A\ar[r]^-{x'} & B \ar[r]^-{y'} & C .   }
$$  A realization $\mathfrak{s}$ of $\mathbb{E}$ is said to be {\em additive} if it satisfies the following conditions:

(a) For any $A,~C\in\mathscr{C}$, the split $\mathbb{E}$-extension $0\in\mathbb{E}(C,A)$ satisfies $\mathfrak{s}(0)=0$.

(b) $\mathfrak{s}(\delta\oplus\delta')=\mathfrak{s}(\delta)\oplus\mathfrak{s}(\delta')$ for any pair of $\mathbb{E}$-extensions $\delta$ and $\delta'$.
\end{definition}

Let $\mathfrak{s}$ be an additive realization of $\mathbb{E}$. If $\mathfrak{s}(\delta)=[A\stackrel{x}{\longrightarrow}B\stackrel{y}{\longrightarrow}C]$, then the sequence $A\stackrel{x}{\longrightarrow}B\stackrel{y}{\longrightarrow}C$ is called a {\em conflation}, $x$ is called an {\em inflation} and $y$ is called a {\em deflation}.
In this case, we say that $A\stackrel{x}{\longrightarrow}B\stackrel{y}{\longrightarrow}C\stackrel{\delta}\dashrightarrow$ is an $\mathbb{E}$-triangle.
We will write $A=\cocone(y)$ and $C=\cone(x)$ if necessary. We say an $\mathbb{E}$-triangle is {\em splitting} if it
realizes 0.

\begin{definition}
(\cite[Definition 2.12]{Na})\label{F}
We call the triplet $(\mathscr{C}, \mathbb{E},\mathfrak{s})$ an {\em extriangulated category} if it satisfies the following conditions:\\
$\rm(ET1)$ $\mathbb{E}$: $\mathscr{C}^{op}\times\mathscr{C}\rightarrow Ab$ is a biadditive functor.\\
$\rm(ET2)$ $\mathfrak{s}$ is an additive realization of $\mathbb{E}$.\\
$\rm(ET3)$ Let $\delta\in\mathbb{E}(C,A)$ and $\delta'\in\mathbb{E}(C',A')$ be any pair of $\mathbb{E}$-extensions, realized as
$\mathfrak{s}(\delta)=[A\stackrel{x}{\longrightarrow}B\stackrel{y}{\longrightarrow}C]$, $\mathfrak{s}(\delta')=[A'\stackrel{x'}{\longrightarrow}B'\stackrel{y'}{\longrightarrow}C']$. For any commutative square in $\mathscr{C}$
$$\xymatrix{
  A \ar[d]_{a} \ar[r]^{x} & B \ar[d]_{b} \ar[r]^{y} & C \\
  A'\ar[r]^{x'} &B'\ar[r]^{y'} & C'}$$
there exists a morphism $(a,c)$: $\delta\rightarrow\delta'$ which is realized by $(a,b,c)$.\\
$\rm(ET3)^{op}$~Dual of $\rm(ET3)$.\\
$\rm(ET4)$~Let $\delta\in\mathbb{E}(D,A)$ and $\delta'\in\mathbb{E}(F,B)$ be $\mathbb{E}$-extensions realized by
$A\stackrel{f}{\longrightarrow}B\stackrel{f'}{\longrightarrow}D$ and $B\stackrel{g}{\longrightarrow}C\stackrel{g'}{\longrightarrow}F$, respectively.
Then there exist an object $E\in\mathscr{C}$, a commutative diagram
\begin{equation}\label{2.1}
\xymatrix{
  A \ar@{=}[d]\ar[r]^-{f} &B\ar[d]_-{g} \ar[r]^-{f'} & D\ar[d]^-{d} \\
  A \ar[r]^-{h} & C\ar[d]_-{g'} \ar[r]^-{h'} & E\ar[d]^-{e} \\
   & F\ar@{=}[r] & F   }
\end{equation}
in $\mathscr{C}$, and an $\mathbb{E}$-extension $\delta''\in \mathbb{E}(E,A)$ realized by $A\stackrel{h}{\longrightarrow}C\stackrel{h'}{\longrightarrow}E$, which satisfy the following compatibilities:\\
$(\textrm{i})$ $D\stackrel{d}{\longrightarrow}E\stackrel{e}{\longrightarrow}F$ realizes $\mathbb{E}(F,f')(\delta')$,\\
$(\textrm{ii})$ $\mathbb{E}(d,A)(\delta'')=\delta$,\\
$(\textrm{iii})$ $\mathbb{E}(E,f)(\delta'')=\mathbb{E}(e,B)(\delta')$.\\
$\rm(ET4)^{op}$ Dual of $\rm(ET4)$.
\end{definition}
From now on, we assume that $\mathscr{C}$ is always an extriangulated category. An object $P$ in $\mathscr{C}$ is called {\em projective} if for any conflation $A\stackrel{x}{\longrightarrow}B\stackrel{y}{\longrightarrow}C$ and any morphism $c: P\rightarrow C$, there exists a morphism $b: P\rightarrow B$ such that $yb=c$. We denote the full subcategory of projective objects in $\mathscr{C}$ by $\mathcal{P}$. Dually, the {\em injective} objects are defined, and the full subcategory of injective objects in $\mathscr{C}$ is denoted by $\mathcal{I}$. We say that $\mathscr{C}$ {\em has enough projective objects} if for any object $M\in\mathscr{C}$, there exists an $\mathbb{E}$-triangle $A\stackrel{}{\longrightarrow}P\stackrel{}{\longrightarrow}M\stackrel{}\dashrightarrow$ satisfying $P\in\mathcal{P}$. Dually, we define that $\mathscr{C}$ {\em has enough injective objects}. In particular, if $\mathscr{C}$ is a triangulated category, then $\mathscr{C}$  has enough projective objects and enough injective objects with $\mathcal{P}$ and $\mathcal{I}$ consisting of zero objects.

Recently, for $n>0$, the higher positive extensions $\mathbb{E}^n(-,-)$ in an extriangulated category have been defined in \cite{Go}. If $\mathscr{C}$ has enough projective objects or enough injective objects, these $\mathbb{E}^n$ are isomorphic to those defined in \cite{LN} (cf. \cite[Remark 3.4]{Go}). By convention, $\mathbb{E}^0=\Hom$.

\begin{proposition}$($\cite[Theorem 3.5]{Go}$)$ Let $A\stackrel{}{\longrightarrow}B\stackrel{}{\longrightarrow}C\stackrel{}\dashrightarrow$ be an $\mathbb{E}$-triangle in $\mathscr{C}$.

$(1)$ We have a long exact sequence
\begin{equation*}\begin{split}&\Hom(C,-)\rightarrow\Hom(B,-)\rightarrow\Hom(A,-)\rightarrow\mathbb{E}^{1}(C,-)\rightarrow\cdots\\
&\cdots\rightarrow\mathbb{E}^{n-1}(A,-)\rightarrow\mathbb{E}^{n}(C,-)\rightarrow\mathbb{E}^{n}(B,-)\rightarrow\mathbb{E}^{n}(A,-)\rightarrow\cdots,\end{split}\end{equation*}

$(2)$ Dually, we have a long exact sequence
\begin{equation*}\begin{split}&\Hom(-,A)\rightarrow\Hom(-,B)\rightarrow\Hom(-,C)\rightarrow\mathbb{E}^{1}(-,A)\rightarrow\cdots\\
&\cdots\rightarrow\mathbb{E}^{n-1}(-,C)\rightarrow\mathbb{E}^{n}(-,A)\rightarrow\mathbb{E}^{n}(-,B)\rightarrow\mathbb{E}^{n}(C,-)\rightarrow\cdots.\end{split}\end{equation*}
\end{proposition}

\section{Right locally homologically finite cases}
In this section, we assume that $\mathscr{C}$  satisfies the following conditions:

$(1)$ $\dim_{k}\mathbb{E}^{i}(X,Y)<\infty$   for any $X,Y\in\mathscr{C}$ and $i\geq0$.

$(2)$ $\mathscr{C}$ is {\em right locally homologically finite}, i.e., $\sum_{i\geq0}\dim_{k}\mathbb{E}^{i}(X,Y)<\infty$.

Given $X,Y,L\in\mathscr{C}$, let $$W_{XY}^{L}:=\{(f,g,\delta)~|~X\stackrel{f}{\longrightarrow}L\stackrel{g}{\longrightarrow}Y\stackrel{\delta}\dashrightarrow~\text{is~an}~\mathbb{E}\text{-triangle}\}.$$
For any $(f,g,\delta)\in W_{XY}^{L}$,
consider the following commutative diagram
\begin{equation*}\label{act}
\xymatrix{
 &  X\ar[d]^{x} \ar[r]^-{f} & L\ar[d]^{l} \ar[r]^-{g} & Y\ar[d]^-{y} \ar@{-->}[r]^-{\delta}   &  \\
 &  X \ar[r]^-{\overline{f}} & L \ar[r]^-{\overline{g}} & Y &  }
\end{equation*}
with $(x,l,y)\in \Aut X\times\Aut L\times\Aut Y$.  By \cite[Proposition 3.7]{Na}, $$X\stackrel{\overline{f}}{\longrightarrow}L\stackrel{\overline{g}}{\longrightarrow}Y\stackrel{{y^{-1}}^{\ast}x_{\ast}\delta}\dashrightarrow$$ is an $\mathbb{E}$-triangle. Hence, for any $(x,l,y)\in \Aut X\times\Aut L\times\Aut Y$, we have a map
 $$\phi_{xly}:W_{XY}^{L}\longrightarrow W_{XY}^{L};~(f,g,\delta)\mapsto(lfx^{-1},ygl^{-1},{y^{-1}}^{\ast}x_{\ast}\delta).$$

We need to consider several group actions on the set $W_{XY}^{L}$. Now, let us give these group actions and the corresponding orbit sets:

The action of $ \Aut X$ on $W_{XY}^{L}$ has the orbit set
$$V_{\overline{X}Y}^{L}=\{(f,g,\delta)_{X}~|~(f,g,\delta)\in W_{XY}^{L}\},$$
where $$(f,g,\delta)_{X}=\{\phi_{x11}((f,g,\delta))~|~x\in\Aut X\}.$$
Similarly, the actions of $\Aut L$ and $\Aut Y$ on $W_{XY}^{L}$ have the orbit sets $V_{XY}^{\overline{L}}$ and $V_{X\overline{Y}}^{{L}}$, respectively.

The action of $ \Aut L\times \Aut Y$ on $W_{XY}^{L}$  has the orbit set
$$V_{X\overline{Y}}^{\overline{L}}=\{(f,g,\delta)^{\wedge}~|~(f,g,\delta)\in W_{XY}^{L}\},$$
where $$(f,g,\delta)^{\wedge}=\{\phi_{1ly}((f,g,\delta))~|~(l,y)\in\Aut L\times \Aut Y\}.$$

The action of $ \Aut X\times \Aut L$  on $W_{XY}^{L}$ has the orbit set
$$V_{\overline{X}Y}^{\overline{L}}=\{(f,g,\delta)^{\vee}~|~(f,g,\delta)\in W_{XY}^{L}\},$$
where $$(f,g,\delta)^{\vee}=\{\phi_{xl1}((f,g,\delta))\mid(x,l)\in\Aut X\times \Aut L\}.$$

In what follows, for the simplicity of notation, we write $(-,-)$ as $\Hom(-,-)$ in $\mathscr{C}$.

Given $X,Y\in\mathscr{C}$, we set
$$\{X,Y\}':=\prod_{i>0}|\mathbb{E}^{i}(X,Y)|^{(-1)^{i}}$$
and
$$[X,Y]:=\frac{1}{\{X,Y\}'|(X,Y)|}.$$
Since $\mathscr{C}$ is right locally homologically finite, $\{X,Y\}'<\infty$ and $[X,Y]<\infty$.

\begin{lemma}\label{1}Let $(f,g,\delta)\in W_{XY}^{L}$.

$(1)$ We have that
$$|\Im(g,L)|=\frac{[X,L]}{[L,L]\{Y,L\}'}$$
and
$$|\Im(Y,g)|=\frac{\{Y,X\}'}{\{Y,L\}'[Y,Y]}.$$

$(2)$ We have that
$$|\Im(L,f)|=\frac{[L,Y]}{[L,L]\{L,X\}'}$$
and
$$|\Im(f,X)|=\frac{\{Y,X\}'}{\{L,X\}'[X,X]}.$$
\end{lemma}

\begin{proof}Applying the functor $\Hom(-,L)$ to the $\mathbb{E}$-triangle $(f,g,\delta)$, we get the  exact sequence
$$ \Hom(Y,L)\stackrel{(g,L)}\rightarrow\Hom(L,L)\rightarrow\Hom(X,L)\rightarrow\mathbb{E}^{1}(Y,L)\rightarrow\cdots.$$
Then
\begin{align*}
|\Im(g,L)|&=\frac{|(L,L)|}{|(X,L)|}\frac{\prod_{i>0}|\mathbb{E}^{i}(L,L)|^{(-1)^{i}}}{\prod_{i>0}|\mathbb{E}^{i}(Y,L)|^{(-1)^{i}}\prod_{i>0}|\mathbb{E}^{i}(X,L)|^{(-1)^{i}}}\\
&=\frac{|(L,L)|}{|(X,L)|}\frac{\{L,L\}'}{\{Y,L\}'\{X,L\}'}\\
&=\frac{[X,L]}{[L,L]\{Y,L\}'}.
\end{align*}

Similarly,  applying the functor $\Hom(Y,-)$ to the $\mathbb{E}$-triangle $(f,g,\delta)$, we get the  exact sequence
$$ \Hom(Y,X)\rightarrow\Hom(Y,L)\stackrel{(Y,g)}\rightarrow\Hom(Y,Y)\rightarrow\mathbb{E}^{1}(Y,X)\rightarrow\cdots.$$
Then
 $$|\Im(Y,g)|=|(Y,Y)|\frac{\{Y,X\}'\{Y,Y\}'}{\{Y,L\}'}=\frac{\{Y,X\}'}{\{Y,L\}'[Y,Y]}.$$
The proof of (2) is similar.
\end{proof}

For $(f,g,\delta)\in W_{XY}^{L}$, we set
$$\mathbf{G}_{L}(f,g,\delta)_{Y}=\{l\in \Aut L~|~(f,g,\delta)_{Y}=(lf,gl^{-1},\delta)_{Y}\}.$$
That is, $\mathbf{G}_{L}(f,g,\delta)_{Y}$ is the stablizer of $(f,g,\delta)_{Y}\in V_{X\overline{Y}}^{L}$ under the action of $\Aut L$.
Dually, we set
$$\mathbf{G}_{Y}(f,g,\delta)_{L}=\{y\in \Aut Y~|~(f,g,\delta)_{L}=(f,yg,{y^{-1}}^{\ast}\delta)_{L}\}.$$

Denote by $\ind(\mathscr{C})$ the set of isomorphism classes of indecomposable objects in $\mathscr{C}$. For any morphism $g:L\rightarrow Y$, by \cite[Lemma 2.2]{XF}, there exists the decomposition $L=L_{1}\oplus L_{2}$, $Y=Y_{1}\oplus Y_{2}$ and  $(l,y)\in\Aut L\times\Aut Y$ such that
$$ygl=\begin{pmatrix} n_{11}&0\\0&n_{22}\end{pmatrix}:L_{1}\oplus L_{2}\rightarrow Y_{1}\oplus Y_{2}$$
with $n_{11}$ being an isomorphism and
$$n_{22}\in \rad\Hom(L_{2},Y_{2}):=\{r\in\Hom(L_{2},Y_{2})~|~r_{1}rr_{2} \text{~is not an isomorphism for any}$$ $$~r_{1}:Y_{2}\rightarrow A~\text{and}~r_{2}:A\rightarrow L_{2}~\text{with}~A\in \ind(\mathscr{C})\}.$$
We may denote $L_{1}$ and $Y_{1}$ by $L_{g}$ and $_{g}Y$, respectively. For $f:X\rightarrow Y$, we set
$$\Hom(Y,Z)f:=\{gf~|~g\in\Hom(Y,Z)\}$$
and
$$f\Hom(Z,X):=\{fh~|~h\in\Hom(Z,X)\}.$$

\begin{lemma}\label{3} Let $(f,g,\delta)\in W_{XY}^{L}$.

$(1)$ $1-\mathbf{G}_{L}(f,g,\delta)_{Y}=\{l\in \End L~|~l\in \Im(g,L)~\text{and}~1-l\in \Aut L\}$.

$(2)$ If $(f,g,\delta)^{\wedge}=(f_{1},g_{1},\delta_{1})^{\wedge}\in V_{X\overline{Y}}^{\overline{L}}$. Then $$|\mathbf{G}_{L}(f,g,\delta)_{Y}|=|\mathbf{G}_{L}(f_{1},g_{1},\delta_{1})_{Y}|.$$

$(3)$ We have that
$$|\mathbf{G}_{L}(f,g,\delta)_{Y}|=\frac{|\Im(g,L)||\Aut L_{g}|}{|\End L_{g}|}.$$



$(4)$ We have that
$$|\mathbf{G}_{Y}(f,g,\delta)_{L}|=\frac{|\Im(Y,g)||\Aut{_{g}Y}|}{|\End _{g}Y|}.$$
\end{lemma}
\begin{proof} (1) It is straightforward that $l\in\mathbf{G}_{L}(f,g,\delta)_{Y}\Leftrightarrow l\in\Aut L$ and $f=lf\Leftrightarrow l\in\Aut L$ and  $1-l=tg$ for some $t:Y\rightarrow L\Leftrightarrow l\in\Aut L$ and $1-l\in \Im(g,L)$.

(2) By hypothesis, there exists a commutative diagram
\begin{equation*}
\xymatrix{
 &  X\ar@{=}[d] \ar[r]^-{f} & L\ar[d]^{l} \ar[r]^-{g} & Y\ar[d]^{y} \ar@{-->}[r]^-{\delta}   &  \\
 &  X \ar[r]^-{f_{1}} & L \ar[r]^-{g_{1}} & Y \ar@{-->}[r]^-{\delta_{1}} &  }
\end{equation*}
with $(l,y)\in \Aut L\times\Aut Y$. Consider the map
$$\alpha:\mathbf{G}_{L}(f,g,\delta)_{Y}\rightarrow\mathbf{G}_{L}(f_{1},g_{1},\delta_{1})_{Y};~l'\mapsto ll'l^{-1}.$$
By (1), it is easy to see that $\alpha$ is a bijection.

(3) By \cite[Lemma 2.2]{XF}, there exists the decomposition $L=L_{g}\oplus L_{2}$, $Y=Y_{1}\oplus Y_{2}$ and  $(l,y)\in\Aut L\times\Aut Y$ such that
$$ygl=\begin{pmatrix} n_{11}&0\\0&n_{22}\end{pmatrix}:L_{g}\oplus L_{2}\rightarrow Y_{1}\oplus Y_{2}$$
with $n_{11}$ being an isomorphism and $n_{22}\in \rad\Hom(L_{2},Y_{2})$. Set $f'=l^{-1}f$, $g'=ygl$ and $\delta'={y^{-1}}^{\ast}\delta$.
Then $(f,g,\delta)^{\wedge}=(f',g',\delta')^{\wedge}$ in $V_{X\overline{Y}}^{\overline{L}}$.

If $l\in \Im(g',L)$, i.e., there exists a morphism
\begin{equation*}t=\begin{pmatrix} t_{11}&t_{12}\\t_{21}&t_{22}\end{pmatrix}:Y_{1}\oplus Y_{2}\rightarrow L_{g}\oplus L_{2} \end{equation*}
such that
\begin{equation}\label{juzhen}{1-l=1-tg'=\begin{pmatrix} 1-t_{11}n_{11}&-t_{12}n_{22}\\-t_{21}n_{11}&1-t_{22}n_{22}\end{pmatrix}.}\end{equation}
Since $n_{22}\in \rad\Hom(L_{2},Y_{2})$, we have that $1-t_{22}n_{22}\in \Aut L_{2}$. By applying some elementary transformations on the matrix in (\ref{juzhen}), we obtain that $1-l\in \Aut L$ if and only if $1-t_{11}n_{11}\in \Aut L_{g}$. By (1) and (2), we have that
\begin{align*}
|\mathbf{G}_{L}(f,g,\delta)_{Y}|&=|\mathbf{G}_{L}(f',g',\delta')_{Y}|\\
&=|1-\mathbf{G}_{L}(f',g',\delta')_{Y}|\\
&=|\{l\in \End L~|~l\in \Im(g',L)~\text{and}~1-l\in \Aut L\}|\\
&=|\Aut L_{g}||\Hom(Y_{1},L_{2})n_{11}||\Hom(Y_{2},L_{g})n_{22}||\Hom(Y_{2},L_{2})n_{22}|\\
&=\frac{|\Im(g',L)||\Aut L_{g}|}{|\End L_{g}|}\\
&=\frac{|\Im(g,L)||\Aut L_{g}|}{|\End L_{g}|}.
\end{align*}



(4) Note that $y\in\mathbf{G}_{Y}(f,g,\delta)_{L}\Leftrightarrow y\in\Aut Y$ and $\delta={y^{-1}}^{\ast}\delta\Leftrightarrow y\in\Aut Y$ and $ 1-y^{-1}=gt$ for some $t:Y\rightarrow L\Leftrightarrow y\in\Aut Y$ and $ 1-y^{-1}\in\Im(Y,g)$. Then it is proved by the analogous arguments as those for proving (3).
\end{proof}

\begin{lemma}\label{4} Given $X,Y,L\in\mathscr{C}$, consider the surjection
$$\sigma_{1}:V_{X\overline{Y}}^{L}\rightarrow V_{X\overline{Y}}^{\overline{L}};~(f,g,\delta)_{Y}\mapsto(f,g,\delta)^{\wedge}.$$

$(1)$ We have that
$$\sigma_{1}^{-1}((f,g,\delta)^{\wedge})=\{(lf,gl^{-1},\delta)_{Y}~|~l\in\Aut L\}.$$

$(2)$ We have that
$$|V_{X\overline{Y}}^{L}|=\sum_{(f,g,\delta)\in V_{X\overline{Y}}^{\overline{L}}}\frac{|\Aut L||\End L_{g}|}{|\Im(g,L)||\Aut L_{g}|}.$$

$(3)$ We have that
$$|V_{XY}^{\overline{L}}|=\sum_{(f,g,\delta)\in V_{X\overline{Y}}^{\overline{L}}}\frac{|\Aut Y||\End _{g}Y|}{|\Im(Y,g)||\Aut{_{g}Y}|}.$$

\end{lemma}
\begin{proof} (1) Assume that $(f_{1},g_{1},\delta_{1})_{Y}\in\sigma_{1}^{-1}((f,g,\delta)^{\wedge})$, then $(f_{1},g_{1},\delta_{1})^{\wedge}=(f,g,\delta)^{\wedge}$. Hence, we have a commutative diagram
\begin{equation*}
\xymatrix{
 &  X\ar@{=}[d] \ar[r]^-{f} & L\ar[d]^{l} \ar[r]^-{g} & Y\ar[d]^{y} \ar@{-->}[r]^-{\delta}   &  \\
 &  X \ar[r]^-{f_{1}} & L \ar[r]^-{g_{1}} & Y \ar@{-->}[r]^-{\delta_{1}} &  }
\end{equation*}
with $(l,y)\in \Aut L\times\Aut Y$. It follows that $(f_{1},g_{1},\delta_{1})=(lf,ygl^{-1},{y^{-1}}^{\ast}\delta)$ and $(f_{1},g_{1},\delta_{1})_{Y}=(lf,gl^{-1},\delta)_{Y}$. Conversely, it is clear that $(lf,gl^{-1},\delta)^{\wedge}=(f,g,\delta)^{\wedge}$ for any $l\in\Aut L$.

(2)  Since $\sigma_{1}$ is surjective, by Lemma \ref{3}(3) and (1), we have that
\begin{align*}
|V_{X\overline{Y}}^{L}|&=\sum_{(f,g,\delta)\in V_{X\overline{Y}}^{\overline{L}}}|\sigma_{1}^{-1}((f,g,\delta)^{\wedge})|\\
&=\sum_{(f,g,\delta)\in V_{X\overline{Y}}^{\overline{L}}}\frac{|\Aut L|}{|\mathbf{G}_{L}(f,g,\delta)_{Y}|}\\
&=\sum_{(f,g,\delta)\in V_{X\overline{Y}}^{\overline{L}}}\frac{|\Aut L||\End L_{g}|}{|\Im(g,L)||\Aut L_{g}|}.
\end{align*}

(3) It is proved by the analogous arguments as those for proving (2).
\end{proof}
Similarly, consider the surjections
$$\sigma_{2}:V_{\overline{X}Y}^{L}\rightarrow V_{\overline{X}Y}^{\overline{L}};~(f,g,\delta)_{X}\mapsto(f,g,\delta)^{\vee},$$
and
$$\sigma_{3}:V_{XY}^{\overline{L}}\rightarrow V_{\overline{X}Y}^{\overline{L}};~(f,g,\delta)_{L}\mapsto(f,g,\delta)^{\vee}.$$
\begin{lemma}\label{FZ} Given $X,Y,L\in\mathscr{C}$, we have that
$$|V_{\overline{X}Y}^{L}|=\sum_{(f,g,\delta)\in V_{\overline{X}Y}^{\overline{L}}}\frac{|\Aut L||\End _{f}L|}{|\Im(L,f)||\Aut _{f}L|}$$
and
$$|V_{XY}^{\overline{L}}|=\sum_{(f,g,\delta)\in V_{\overline{X}Y}^{\overline{L}}}\frac{|\Aut X||\End X_{f}|}{|\Im(f,X)||\Aut X_{f}|}.$$
\end{lemma}


We denote by $(X,L)_{Y}$ the set consisting of inflations $f:X\rightarrow L$ such that $\cone(f)\cong Y$. Dually, we define $_{X}(L,Y)$. We denote by $\mathbb{E}(Y,X)_{L}$ the set consisting of extensions $\delta\in\mathbb{E}(Y,X)$ such that $\mathfrak{s}(\delta)=[X\stackrel{f}{\longrightarrow}L\stackrel{g}{\longrightarrow}Y]$.

\begin{lemma}\label{11} Given $X,Y,L\in\mathscr{C}$, the maps
$$\varphi_{1}:V_{X\overline{Y}}^{L}\rightarrow{(X,L)_{Y}};~(f,g,\delta)_{Y}\mapsto f,$$

$$\varphi_{2}:V_{XY}^{\overline{L}}\rightarrow\mathbb{E}(Y,X)_{L};~(f,g,\delta)_{Y}\mapsto \delta,$$
and
$$\varphi_{3}:V_{\overline{X}Y}^{L}\rightarrow{_{X}(L,Y)};~(f,g,\delta)_{Y}\mapsto g$$
are  bijections.

\end{lemma}
\begin{proof} Recall that the orbit
$$(f,g,\delta)_{Y}=\{(f,yg,{y^{-1}}^{\ast}\delta)\mid y\in\Aut Y\}.$$
Thus $\varphi$ is a well-defined surjection. Now, assume that $\varphi_{1}((f,g,\delta)_{Y})=\varphi_{1}((f_{1},g_{1},\delta_{1})_{Y})$. Then, by $\rm(ET3)$ and \cite[Corollary 3.6]{Na},  we have the commutative diagram
\begin{equation*}\label{act}
\xymatrix{
 &  X\ar@{=}[d] \ar[r]^-{f} & L\ar@{=}[d] \ar[r]^-{g} & Y\ar@{-->}[d]^{y} \ar@{-->}[r]^-{\delta}   &  \\
 &  X \ar[r]^-{f_{1}} & L \ar[r]^-{g_{1}} & Y \ar@{-->}[r]^-{\delta_{1}} &  }
\end{equation*}
with $y\in \Aut Y$. It follows that $(f,g,\delta)_{Y}=(f_{1},g_{1},\delta_{1})_{Y}$. Similarly, we can prove that $\varphi_{2}$ and $\varphi_{3}$ are bijections.
\end{proof}

Now, we have the following key formula.
\begin{proposition}\label{main}  Given $X,Y,L\in\mathscr{C}$, we have that
$$\sum_{(f,g,\delta)\in V_{X\overline{Y}}^{\overline{L}}}\frac{|\End _{g}Y|}{|\Aut _{g}Y|}\{Y,L\}'=\frac{[X,L]}{[L,L]}\frac{|(X,L)_{Y}|}{|\Aut L|}=\frac{\{Y,X\}'}{[Y,Y]}\frac{|\mathbb{E}(Y,X)_{L}|}{|\Aut Y|}.$$
\end{proposition}
\begin{proof} By Lemma \ref{4} and Lemma \ref{1}, we obtain that
\begin{align*}
|\Im(g,L)|\frac{|V_{X\overline{Y}}^{L}|}{|\Aut L|}&=\sum_{(f,g,\delta)\in V_{X\overline{Y}}^{\overline{L}}}\frac{|\End L_{g}|}{|\Aut L_{g}|}\\
&=\sum_{(f,g,\delta)\in V_{X\overline{Y}}^{\overline{L}}}\frac{|\End _{g}Y|}{|\Aut _{g}Y|}\\
&=|\Im(Y,g)|\frac{|V_{XY}^{\overline{L}}|}{|\Aut Y|}.
\end{align*}
Using Lemma \ref{1}, we have that
$$\frac{[X,L]}{[L,L]\{Y,L\}'}\frac{|V_{X\overline{Y}}^{L}|}{|\Aut L|}=\frac{\{Y,X\}'}{\{Y,L\}'[Y,Y]}\frac{|V_{XY}^{\overline{L}}|}{|\Aut Y|}.$$
By Lemma \ref{11}, we complete the proof.
\end{proof}

Dually, we also have the following formula.
\begin{proposition}\label{S} Given $X,Y,L\in\mathscr{C}$, we have that
$$\sum_{(f,g,\delta)\in V_{\overline{X}Y}^{\overline{L}}}\frac{|\End X_{f}|}{|\Aut X_{f}|}\{L,X\}'=\frac{[L,Y]}{[L,L]}\frac{|_{X}(L,Y)|}{|\Aut L|}=\frac{\{Y,X\}'}{[X,X]}\frac{|\mathbb{E}(Y,X)_{L}|}{|\Aut X|}.$$
\end{proposition}
\begin{proof} Similarly, using Lemma \ref{1}, Lemma \ref{FZ} and Lemma \ref{11}, we complete the proof.
\end{proof}

Let us state the  main result in this section in the following
\begin{theorem}\label{main2} Given $X,Y,L\in\mathscr{C}$, we have that
$$\frac{|_{X}(L,Y)|}{|\Aut Y|}\frac{[L,Y]}{[Y,Y]}=\frac{|(X,L)_{Y}|}{|\Aut X|}\frac{[X,L]}{[X,X]}.$$

\end{theorem}
\begin{proof} By Proposition \ref{main}, we have that
$$[X,L]|(X,L)_{Y}|=\frac{\{Y,X\}'}{[Y,Y]}\frac{|\mathbb{E}(Y,X)_{L}|}{|\Aut Y|}[L,L]|\Aut L|.$$
By Proposition \ref{S}, we have that
$$[L,Y]|_{X}(L,Y)|=\frac{\{Y,X\}'}{[X,X]}\frac{|\mathbb{E}(Y,X)_{L}|}{|\Aut X|}[L,L]|\Aut L|.$$
Thus, we conclude that
$$\frac{[X,L]|(X,L)_{Y}|}{[L,Y]|_{X}(L,Y)|}=\frac{[X,X]|\Aut X|}{[Y,Y]|\Aut Y|}.$$
This finishes the proof.
\end{proof}

Let $X,Y,L\in\mathscr{C}$ and $\mathcal{X}$ be a subset of $W_{XY}^{L}$. Then the actions of $\Aut X$ and $\Aut Y$ on $W_{XY}^{L}$ naturally induces the actions on $\mathcal{X}$. We denote the orbit sets of $\mathcal{X}$ under the actions of $\Aut X$ and $\Aut Y$ by $\mathcal{X}_{X}$ and $\mathcal{X}_{Y}$, respectively. The following observation is useful for the next section.
\begin{lemma}\label{TT} Let $X,Y,L\in\mathscr{C}$ and $\mathcal{X}$ be a subset of $W_{XY}^{L}$. Then we have that
$$\frac{|\mathcal{X}_{X}|}{|\Aut Y|}\frac{[L,Y]}{[Y,Y]}=\frac{|\mathcal{X}_{Y}|}{|\Aut X|}\frac{[X,L]}{[X,X]}.$$

\end{lemma}
\begin{proof} It is proved by the analogous arguments as those for proving Theorem \ref{main2}.
\end{proof}

\section{Hall algebras of extriangulated categories, ${\rm{\RNum{1}}}$}
In this section, keep the conditions on $\mathscr{C}$ in  Section 3, we define the {Hall algebra} of extriangulated categories using the formula in Theorem \ref{main2}.
\begin{definition}\label{H1} The Hall algebra $\mathcal{H}(\mathscr{C})$ of the extriangulated category $\mathscr{C}$ is a $\mathbb{Q}$-space with the basis $\{u_{[X]}~|~X\in \mathscr{C}\}$ and the multiplication defined by
\begin{equation*}u_{[X]}\ast u_{[Y]}=\sum_{[L]}F_{XY}^{L}u_{[L]}\end{equation*}
where $$F_{XY}^{L}:=\frac{|_{X}(L,Y)|}{|\Aut Y|}\frac{[L,Y]}{[Y,Y]}=\frac{|(X,L)_{Y}|}{|\Aut X|}\frac{[X,L]}{[X,X]}.$$
\end{definition}

Let us state the main result in this paper as the following
\begin{theorem}\label{main3} The Hall algebra $\mathcal{H}(\mathscr{C})$ of the extriangulated category $\mathscr{C}$ is an associative algebra.
\end{theorem}

Before proving Theorem \ref{main3}, we need some preparations.

We set
$$_{Z}^{L'}(X\oplus M,L)_{Y}=\{(f~g):X\oplus M\rightarrow L~|~\cone(f)\cong Y,\cocone(g)\cong Z~\text{and}~\cocone((f~g))\cong L'\}$$
and
$$_{Z}(L',X\oplus M)_{Y}^{L}=\{(f~g)^{T}:L'\rightarrow X\oplus M~|~\cocone(f)\cong Z,\cone(g)\cong Y, \text{and}~\cone((f~g)^{T})\cong L\}.$$

\begin{lemma}\label{Zh}

$(1)$ The map
$$\tau_{1}:(X,L)_{Y}\times{_{Z}(M,L)}\rightarrow \bigcup_{[L']}{_{Z}^{L'}}(X\oplus M,L)_{Y};~f\times g\mapsto (f~g)$$
is a  bijection.

$(2)$ The map
$$\tau_{2}:{_{Z}(L',X)}\times{(L',M)_{Y}}\rightarrow \bigcup_{[L]}{_{Z}}(L',X\oplus M)_{Y}^{L};~f\times g\mapsto (f~g)^{T}$$
is a  bijection.
\end{lemma}
\begin{proof} (1) By $\rm(ET4)^{op}$, there exists a commutative diagram
\begin{equation*}
\xymatrix{
 &  Z\ar@{=}[d] \ar[r]^-{} & L'\ar[d]^{} \ar[r]^-{} & X\ar[d]^{f} \ar@{-->}[r]   &  \\
 &  Z \ar[r]^-{} & M \ar[r]^-{g} & L \ar@{-->}[r] & . }
\end{equation*}
By the dual of \cite[Corollary 3.16]{Na}, we obtain that $(f~g)\in{_{Z}^{L'}}(X\oplus M,L)_{Y}$. By definition, it is easy to see that $\tau_{1}$ is a bijection. The proof of (2) is similar.
\end{proof}

\begin{lemma}\label{s1}
Given a commutative diagram of $\mathbb{E}$-triangles
$$ \xymatrix{
   X\ar@{=}[d] \ar[r]^-{f} & L\ar[d]^{b} \ar[r]^-{g} & Y\ar[d]^{y} \ar@{-->}[r]^-{\delta}   &  \\
   X \ar[r]^-{f'} & L' \ar[r]^-{g'} & Y' \ar@{-->}[r]^-{\delta'} & . }$$
   Then there exists a morphism $s:Y\rightarrow Y'$ such that the sequence
   $$L\stackrel{(b~g)^T}{\longrightarrow}L'\oplus Y\stackrel{(g'~-s)}{\longrightarrow}Y'\stackrel{f_{\ast}\delta'}\dashrightarrow$$
   is an $\mathbb{E}$-triangle and $s^{\ast}\delta'=\delta$. Dually for the commutative diagram of $\mathbb{E}$-triangles
   $$ \xymatrix{
   X\ar[d]^-{x} \ar[r]^-{f} & L\ar[d]^{b} \ar[r]^-{g} & Y\ar@{=}[d] \ar@{-->}[r]^-{\delta}   &  \\
   X' \ar[r]^-{f'} & L' \ar[r]^-{g'} & Y \ar@{-->}[r]^-{\delta'} & . }$$
\end{lemma}
\begin{proof} By \cite[Proposition 3.15]{Na}, we have the commutative diagram of $\mathbb{E}$-triangles
$$\xymatrix{
  X \ar[d]^{f'} \ar[r]^-{f} & L \ar[d]^-{(l~g)^T} \ar[r]^{g} & Y \ar@{=}[d]\ar@{-->}[r]^{\delta}& \\
  L'\ar[r]^-{(1~0)^T}\ar[d]^{g'} & L'\oplus Y \ar^-{(g'~k)}[d] \ar[r]^{~~(0~1)} &Y \ar@{-->}[r]^{0} &\\
   Y'\ar@{=}[r]\ar@{-->}[d]^{\delta'}& Y' \ar@{-->}[d]^{} & & \\
  &&& }
$$
satisfying $(g'~k)^{\ast}\delta'+(0,1)^{\ast}\delta=0$, in which we have assumed that the middle row is of the form
$L'\stackrel{(1~0)^{T}}{\longrightarrow}L'\oplus Y\stackrel{(0~1)}{\longrightarrow}Y\stackrel{0}\dashrightarrow,$
since $f'_{\ast}\delta=f'_{\ast}y^{\ast}\delta'=y^{\ast}(f'_{\ast}\delta')=0$.

Since $bf=f'=lf$, there exists a morphism $h:Y\rightarrow L'$ such that $b-l=hg$. Thus, we have the following commutative diagram with $s=g'h-k$
$$ \xymatrix{
   L\ar@{=}[d] \ar[r]^-{(b~g)^{T}} & L'\oplus Y\ar[d]^{\tiny\begin{pmatrix} 1&-h\\0&1\end{pmatrix}} \ar[r]^-{(g'~-s)} & Y\ar@{=}[d]     \\
   L \ar[r]_-{(l~g)^{T}} & L'\oplus Y \ar[r]_-{(g'~k)} & Y .}$$
Then we obtain that $L\stackrel{(b~g)^T}{\longrightarrow}L'\oplus Y\stackrel{(g'~-s)}{\longrightarrow}Y'\stackrel{f_{\ast}\delta'}\dashrightarrow$
is an $\mathbb{E}$-triangle, which also implies that $sg=g'b$.  
Note that $s^{\ast}\delta'=(g'h-k)^{\ast}\delta'=h^{\ast}(g'^{\ast}\delta')-k^{\ast}\delta'=-k^{\ast}\delta'$ and
$k^{\ast}\delta'+\delta={(0~1)^{T}}^{\ast}((g'~k)^{\ast}\delta'+(0,1)^{\ast}\delta)=0,$ we get that $s^{\ast}\delta'=\delta$.
So, we complete the proof.
\end{proof}

\begin{lemma}\label{Q} Let
$$L'\stackrel{(f~g)^{T}}{\longrightarrow}X\oplus M\stackrel{(f'~g')}{\longrightarrow}L\stackrel{}\dashrightarrow $$
be an $\mathbb{E}$-triangle. Then $(f,g)\in{_{Z}(L',X)}\times{(L',M)_{Y}}$ if and only if $(f',g')\in(X,L)_{Y}\times{_{Z}(M,L)}$.
\end{lemma}
\begin{proof} Assume that $(f,g)\in{_{Z}(L',X)}\times{(L',M)_{Y}}$. By $\rm(ET4)$, there exists a commutative diagram of $\mathbb{E}$-triangles
$$\xymatrix{
  Z \ar@{=}[d] \ar[r] & L' \ar[d]^{g} \ar[r]^{f} & X \ar[d]^{f''}  \\
  Z  \ar[r] & M \ar[d] \ar[r]^{g''} & L'' \ar[d] \\
  & Y  \ar@{=}[r] & Y.   }
$$
By Lemma \ref{s1}, there exists a morphism $s:X\rightarrow L''$ such that
$$L'\stackrel{(f~g)^{T}}{\longrightarrow}X\oplus M\stackrel{(-s~g'')}{\longrightarrow}L''\stackrel{}\dashrightarrow $$
is an $\mathbb{E}$-triangle. Hence, by \cite[Corollary 3.6]{Na},  we have the commutative diagram
\begin{equation*}\label{act}
\xymatrix{
 &  L'\ar@{=}[d] \ar[r]^-{(f~g)^{T}} & X\oplus M\ar@{=}[d] \ar[r]^-{(f'~g')} & L\ar@{-->}[d]^{l} \ar@{-->}[r]^-{\delta}   &  \\
 &  L' \ar[r]^-{(f~g)^{T}} & X\oplus M \ar[r]^-{(-s~g'')} &  L'' \ar@{-->}[r]^-{} &  }
\end{equation*}
with $l$ being an isomorphism. It follows that $g'=l^{-1}g''$ and $\cocone(g')\cong\cocone(g'')=Z$. Similarly, by \cite[Proposition 1.20]{LN}, there exists a morphism $t:M\rightarrow L''$ such that

$$L'\stackrel{(-f~g)^{T}}{\longrightarrow}X\oplus M\stackrel{(f''~t)}{\longrightarrow}L''\stackrel{}\dashrightarrow $$
is an $\mathbb{E}$-triangle. We have the commutative diagram
\begin{equation*}\label{act}
\xymatrix{
 &  L'\ar@{=}[d] \ar[r]^-{(f~g)^{T}} & X\oplus M\ar[d]^-{\tiny\begin{pmatrix} -1&0\\0&1\end{pmatrix}} \ar[r]^-{(f'~g')} & L\ar@{-->}[d]^{l'} \ar@{-->}[r]^-{\delta}   &  \\
 &  L' \ar[r]^-{(-f~g)^{T}} & X\oplus M \ar[r]^-{(f''~t)} &  L'' \ar@{-->}[r]^-{} &  }
\end{equation*}
with $l'$ being an isomorphism. It follows that $f'=l'^{-1}f''$ and $\cone(f')\cong\cone(f'')=Y$. Dually, if  $(f',g')\in(X,L)_{Y}\times{_{Z}(M,L)}$, we also have that $(f,g)\in{_{Z}(L',X)}\times{(L',M)_{Y}}$.
\end{proof}

By Lemma \ref{Q}, we set
\begin{align*}
W(L',X\oplus M,L)_{ZY}&=\{((f~g)^{T},(f'~g'),\delta)\in W_{L'L}^{X\oplus M}~|~(f,g)\in{_{Z}(L',X)}\times{(L',M)_{Y}}\}\\
&=\{((f~g)^{T},(f'~g'),\delta)\in W_{L'L}^{X\oplus M}~|~(f',g')\in(X,L)_{Y}\times{_{Z}(M,L)}\}.
\end{align*}
The actions of $\Aut L$ and $\Aut L'$ on $W(L',X\oplus M,L)_{ZY}$ have the orbit sets denoted by
$$V(L',X\oplus M,\overline{L})_{ZY}=\{((f~g)^{T},(f'~g'),\delta)_{L}~|~((f~g)^{T},(f'~g'),\delta)\in W(L',X\oplus M,L)_{ZY}\}$$
and
$$V(\overline{L'},X\oplus M,L)_{ZY}=\{((f~g)^{T},(f'~g'),\delta)_{L'}~|~((f~g)^{T},(f'~g'),\delta)\in W(L',X\oplus M,L)_{ZY}\}.$$

\begin{lemma}\label{Wa} The maps
$$\varphi_{1}:V(L',X\oplus M,\overline{L})_{ZY}\rightarrow{_{Z}(L',X\oplus M)_{Y}^{L}};~((f~g)^{T},(f'~g'),\delta)_{L}\mapsto (f~g)^{T}$$
and
$$\varphi_{2}:V(\overline{L'},X\oplus M,L)_{ZY}\rightarrow{_{Z}^{L'}}(X\oplus M,L)_{Y};~((f~g)^{T},(f'~g'),\delta)_{L'}\mapsto (f'~g')$$
are  bijections.
\end{lemma}
\begin{proof} The proof is similar to Lemma \ref{11}.
\end{proof}

\begin{proposition}\label{fin} We have that
$$\frac{|{_{Z}^{L'}}(X\oplus M,L)_{Y}|}{|\Aut L|}\frac{[X\oplus M,L]}{[L,L]}=\frac{|{_{Z}(L',X\oplus M)_{Y}^{L}}|}{|\Aut L'|}\frac{[L',X\oplus M]}{[L',L']}.$$
\end{proposition}
\begin{proof} It is proved by Lemma \ref{TT} and Lemma \ref{Wa}.
\end{proof}
Now we are in the position to prove Theorem \ref{main3}.

\textbf{{Proof of Theorem \ref{main3}.}} We need to prove that
$$u_{[Z]}\ast (u_{[X]}\ast u_{[Y]})=(u_{[Z]}\ast u_{[X]})\ast u_{[Y]}$$
for any $X,Y,Z\in\mathscr{C}$. By definition, it is equivalent to proving that
$$\sum_{[L]}F_{XY}^{L}F_{ZL}^{M}=\sum_{[L']}F_{ZX}^{L'}F_{L'Y}^{M}$$
for $M\in\mathscr{C}$. By Lemma \ref{Zh},
\begin{align*}
\sum_{[L]}F_{XY}^{L}F_{ZL}^{M}&=\sum_{[L]}\frac{|(X,L)_{Y}|}{|\Aut X|}\frac{[X,L]}{[X,X]}\frac{|_{Z}(M,L)|}{|\Aut L|}\frac{[M,L]}{[L,L]}\\
&=\frac{1}{|\Aut X|[X,X]}\sum_{[L]}\frac{|(X,L)_{Y}||_{Z}(M,L)|}{|\Aut L|}\frac{[X\oplus M,L]}{[L,L]}\\
&=\frac{1}{|\Aut X|[X,X]}\sum_{[L]}\sum_{[L']}\frac{|{_{Z}^{L'}}(X\oplus M,L)_{Y}|}{|\Aut L|}\frac{[X\oplus M,L]}{[L,L]}.
\end{align*}
Similarly,
\begin{align*}
\sum_{[L']}F_{ZX}^{L'}F_{L'Y}^{M}&=\sum_{[L']}\frac{|_{Z}(L',X)|}{|\Aut X|}\frac{[L',X]}{[X,X]}\frac{|(L',M)_{Y}|}{|\Aut L'|}\frac{[L',M]}{[L',L']}\\
&=\frac{1}{|\Aut X|[X,X]}\sum_{[L']}\frac{|_{Z}(L',X)||(L',M)_{Y}|}{|\Aut L'|}\frac{[L',X\oplus M]}{[L',L']}\\
&=\frac{1}{|\Aut X|[X,X]}\sum_{[L']}\sum_{[L]}\frac{|{_{Z}}(L',X\oplus M)_{Y}^{L}|}{|\Aut L'|}\frac{[L',X\oplus M]}{[L',L']}.
\end{align*}
By Proposition \ref{fin}, we finish the proof.
\fin\\

\section{Left locally homologically finite cases}
According to \cite{Go}, a morphism $f\in\Hom(X,Y)$ is said to be {\em projective} if $\mathbb{E}(f,-)=f^*$ is zero. Dually, a morphism $f\in\Hom(X,Y)$ is {\em injective} if $\mathbb{E}(-,f)=f_*$ is zero. A projective morphism in $\mathscr{C}$ which is also a deflation is called a {\em projective
deflation}. We say that $\mathscr{C}$ {\em has enough projective morphisms} if any $C\in\mathscr{C}$ admits a projective
deflation $G\stackrel{g}{\longrightarrow}C$. Dually, we define that $\mathscr{C}$ {\em has enough injective morphisms}.
We remark that, if $\mathscr{C}$ has enough projective (resp. injective) objects, then it has enough projective (resp. injective) morphisms (cf. \cite[Remark 2.11]{Go}).

In \cite{Go}, for $n>0$, two versions $\mathbb{E}^{-n}_{\rm{\RNum{1}}}$ and $\mathbb{E}^{-n}_{\rm{\RNum{2}}}$ of the higher negative extensions in $\mathscr{C}$ have been defined for the extriangulated categories having enough projective morphisms and enough injective morphisms, respectively. We should remark that
for $n\geq0$, set $\mathbb{E}^{n}_{\rm{\RNum{1}}}=\mathbb{E}^{n}_{\rm{\RNum{2}}}=\mathbb{E}^{n}$.

In this section, we assume that $\mathscr{C}$  satisfies the following conditions:

$(1)$ $\mathscr{C}$ has enough projective morphisms and enough injective morphisms.

$(2)$ $\dim_{k}\mathbb{E}^{i}_{\rm{\RNum{1}}}(X,Y)<\infty$ and $\dim_{k}\mathbb{E}^{i}_{\rm{\RNum{2}}}(X,Y)<\infty$  for any $X,Y\in\mathscr{C}$ and $i\leq0$. 

$(3)$ $\mathscr{C}$ is {\em left locally homologically finite}, i.e., for any $X,Y\in\mathscr{C}$,
$$\sum_{i\leq0}\dim_{k}\mathbb{E}^{i}_{_{\rm{\RNum{1}}}}(X,Y)<\infty~\text{and}~\sum_{i\leq0}\dim_{k}\mathbb{E}^{i}_{_{\rm{\RNum{2}}}}(X,Y)<\infty.$$

We remark that for $n>0$, the bifunctors $\mathbb{E}^{-n}_{\rm{\RNum{1}}}$ and $\mathbb{E}^{-n}_{\rm{\RNum{2}}}$ are not isomorphic (cf. \cite[Section 5.4]{Go}).

\begin{proposition}$($\cite{Go}$)$\label{exact}
Let $A\stackrel{}{\longrightarrow}B\stackrel{}{\longrightarrow}C\stackrel{}\dashrightarrow$ be an $\mathbb{E}$-triangle in $\mathscr{C}$.

$(1)$ We have a long exact sequence
$$\cdots\rightarrow \mathbb{E}^{-1}_{\rm{\RNum{2}}}(A,-)\rightarrow\mathbb{E}^{0}_{\rm{\RNum{2}}}(C,-)\rightarrow\mathbb{E}^{0}_{\rm{\RNum{2}}}(B,-)\rightarrow\mathbb{E}^{0}_{\rm{\RNum{2}}}(A,-)\rightarrow\mathbb{E}_{\rm{\RNum{2}}}^{1}(C,-)\rightarrow\cdots,$$

$(2)$ Dually, we have a long exact sequence
$$\cdots\rightarrow \mathbb{E}^{-1}_{\rm{\RNum{1}}}(-,C)\rightarrow\mathbb{E}^{0}_{\rm{\RNum{1}}}(-,A)\rightarrow\mathbb{E}^{0}_{\rm{\RNum{1}}}(-,B)\rightarrow\mathbb{E}^{0}_{\rm{\RNum{1}}}(-,C)\rightarrow\mathbb{E}^{1}_{\rm{\RNum{1}}}(-,A)\rightarrow\cdots.$$
\end{proposition}

Given $X,Y\in\mathscr{C}$, we set
$$\{X,Y\}_{d}:=\prod_{i>0}|\mathbb{E}_{d}^{-i}(X,Y)|^{(-1)^{i}}$$
where $d\in\{\rm{\RNum{1}},\rm{\RNum{2}}\}$.
Since $\mathscr{C}$ is left locally homologically finite, we have that $\{X,Y\}_{d}<\infty$.

\begin{lemma}\label{5.1}Let $(f,g,\delta)\in W_{XY}^{L}$.

$(1)$ We have that
$$|\Im(g,L)|=|(Y,L)|\frac{\{X,L\}_{\rm{\RNum{2}}}\{Y,L\}_{\rm{\RNum{2}}}}{\{L,L\}_{\rm{\RNum{2}}}}$$
and
$$|\Im(Y,g)|=\frac{|(Y,L)|}{|(Y,X)|}\frac{\{Y,L\}_{\rm{\RNum{1}}}}{\{Y,Y\}_{\rm{\RNum{1}}}\{Y,X\}_{\rm{\RNum{1}}}}.$$

$(2)$ We have that
$$|\Im(L,f)|=|(L,X)|\frac{\{L,Y\}_{\rm{\RNum{1}}}\{L,X\}_{\rm{\RNum{1}}}}{\{L,L\}_{\rm{\RNum{1}}}}$$
and
$$|\Im(f,X)|=\frac{|(L,X)|}{|(Y,X)|}\frac{\{L,X\}_{\rm{\RNum{2}}}}{\{X,X\}_{\rm{\RNum{2}}}\{Y,X\}_{\rm{\RNum{2}}}}.$$
\end{lemma}
\begin{proof} The proof is similar to Lemma \ref{1}.
\end{proof}

\begin{proposition}\label{5.2}Given $X,Y,L\in\mathscr{C}$, we have that
$$\frac{|(X,L)_{Y}|}{|\Aut L|}\frac{\{X,L\}_{\rm{\RNum{2}}}\{Y,L\}_{\rm{\RNum{2}}}}{\{L,L\}_{\rm{\RNum{2}}}}=\frac{|\mathbb{E}(Y,X)_{L}|}{|\Aut Y|}\frac{\{Y,L\}_{\rm{\RNum{1}}}}{\{Y,Y\}_{\rm{\RNum{1}}}|(Y,X)|\{Y,X\}_{\rm{\RNum{1}}}}.$$
\end{proposition}
\begin{proof} By Lemma \ref{4} and Lemma \ref{11}, we have that
$$|\Im(g,L)|\frac{|(X,L)_{Y}|}{|\Aut L|}=|\Im(Y,g)|\frac{|\mathbb{E}(Y,X)_{L}|}{|\Aut Y|}.$$
By Lemma \ref{5.1}, we finish the proof.
\end{proof}
Dually, we also have the following formula.
\begin{proposition}\label{5.3}Given $X,Y,L\in\mathscr{C}$, we have that
$$\frac{|_{X}(L,Y)|}{|\Aut L|}\frac{\{L,Y\}_{\rm{\RNum{1}}}\{L,X\}_{\rm{\RNum{1}}}}{\{L,L\}_{\rm{\RNum{1}}}}=\frac{|\mathbb{E}(Y,X)_{L}|}{|\Aut X|}\frac{\{L,X\}_{\rm{\RNum{2}}}}{\{X,X\}_{\rm{\RNum{2}}}|(Y,X)|\{Y,X\}_{\rm{\RNum{2}}}}.$$
\end{proposition}
Let us state the  main result in
this section in the following, which is a generalization of the case for triangulated categories in \cite{XF}.
\begin{theorem}\label{main5}Given $X,Y,L\in\mathscr{C}$, we have that
$$\frac{|_{X}(L,Y)|}{|\Aut Y|}\frac{\{L,Y\}_{\rm{\RNum{1}}}}{\{Y,Y\}_{\rm{\RNum{1}}}}\hbar_{XY}^{L}=\frac{|(X,L)_{Y}|}{|\Aut X|}\frac{\{X,L\}_{\rm{\RNum{2}}}}{\{X,X\}_{\rm{\RNum{2}}}},$$
where
$$\hbar_{XY}^{L}=\frac{\{Y,X\}_{\rm{\RNum{2}}}}{\{Y,X\}_{\rm{\RNum{1}}}}\frac{\{Y,L\}_{\rm{\RNum{1}}}}{\{Y,L\}_{\rm{\RNum{2}}}}\frac{\{L,X\}_{\rm{\RNum{1}}}}{\{L,X\}_{\rm{\RNum{2}}}}\frac{\{L,L\}_{\rm{\RNum{2}}}}{\{L,L\}_{\rm{\RNum{1}}}}.$$
\end{theorem}
 \begin{proof}  By Proposition \ref{5.2}, we have that
$$|(X,L)_{Y}|\{X,L\}_{\rm{\RNum{2}}}=\frac{|\mathbb{E}(Y,X)_{L}|}{|\Aut Y|}\frac{1}{\{Y,Y\}_{\rm{\RNum{1}}}}\frac{1}{|(Y,X)|\{Y,X\}_{\rm{\RNum{1}}}}\frac{\{Y,L\}_{\rm{\RNum{1}}}}{\{Y,L\}_{\rm{\RNum{2}}}}|\Aut L|\{L,L\}_{\rm{\RNum{2}}}.$$
By Proposition \ref{5.3}, we have that
$$|_{X}(L,Y)|\{L,Y\}_{\rm{\RNum{1}}}=\frac{|\mathbb{E}(Y,X)_{L}|}{|\Aut X|}\frac{1}{\{X,X\}_{\rm{\RNum{2}}}}\frac{1}{|(Y,X)|\{Y,X\}_{\rm{\RNum{2}}}}\frac{\{L,X\}_{\rm{\RNum{2}}}}{\{L,X\}_{\rm{\RNum{1}}}}|\Aut L|\{L,L\}_{\rm{\RNum{1}}}.$$
Thus, we conclude that
$$\frac{|(X,L)_{Y}|\{X,L\}_{\rm{\RNum{2}}}}{|_{X}(L,Y)|\{L,Y\}_{\rm{\RNum{1}}}}=\frac{|\Aut X|\{X,X\}_{\rm{\RNum{2}}}}{|\Aut Y|\{Y,Y\}_{\rm{\RNum{1}}}}\hbar_{XY}^{L}.$$
This finishes the proof.
\end{proof}

Let $\mathscr{C}$ be a triangulated category. By \cite[Proposition 5.4]{Go} and its dual, for $n>0$, we have that $\mathbb{E}^{-n}_{\rm{\RNum{1}}}(X,Y)\cong(X,Y[-n])\cong(X[n],Y)\cong\mathbb{E}^{-n}_{\rm{\RNum{2}}}(X,Y)$. Hence, we have that
$$\{X,Y\}_{\rm{\RNum{1}}}=\{X,Y\}_{\rm{\RNum{2}}},$$ which is denoted by $\{X,Y\}$ in this case.

Recall that in a triangulated category, one denotes by $\Hom(L,Y)_{X[1]}$ the subset of $\Hom(L,Y)$ consisting of morphisms whose cone is isomorphic to $X[1]$.
\begin{corollary}\label{12} $($To\"{e}n's formula$)$ Let $\mathscr{C}$ be a triangulated category. For any $X,Y,L\in\mathscr{C}$, we have that
$$\frac{|(L,Y)_{X[1]}|}{|\Aut Y|}\frac{\{L,Y\}}{\{Y,Y\}}=\frac{|(X,L)_{Y}|}{|\Aut X|}\frac{\{X,L\}}{\{X,X\}}.$$
\end{corollary}
\begin{proof}
In this case, just noting that $|_{X}(L,Y)|=|(L,Y)_{X[1]}|$, we complete the proof.
\end{proof}
Let $\mathscr{C}$ be an exact category. For $d\in\{\rm{\RNum{1}},\rm{\RNum{2}}\}$ and $i<0$, by \cite[Proposition 5.4]{Go} and its dual, we know that $\mathbb{E}^{-i}_{d}(X,Y)=0$. Thus, $\{X,Y\}_{\rm{\RNum{1}}}=\{X,Y\}_{\rm{\RNum{2}}}=1.$
 \begin{corollary}\label{Z} Let $\mathscr{C}$ be an exact category. For any $X,Y,L\in\mathscr{C}$, we have that
$$\frac{|W_{XY}^{L}|}{|\Aut X||\Aut Y|}=\frac{|_{X}(L,Y)|}{|\Aut Y|}=\frac{|(X,L)_{Y}|}{|\Aut X|}.$$
\end{corollary}
\begin{proof} Since each deflation in $\mathscr{C}$ is an epimorphism, the action of $\Aut Y$ on $W_{XY}^L$ is free. So, $$|(X,L)_{Y}|=|V_{X\overline{Y}}^{L}|=\frac{|W_{XY}^{L}|}{|\Aut Y|}.$$
Similarly, $|_{X}(L,Y)|=\frac{|W_{XY}^{L}|}{|{\rm Aut} X|}$.
Thus, we finish the proof.
\end{proof}

\section{Hall algebras of extriangulated categories, ${\rm{\RNum{2}}}$}
In this section, we assume that $\mathscr{C}$ has enough projective objects and enough injective objects, and it satisfies the conditions (2), (3) in Section 5, as well as the following conditions (see \cite[Condition 5.12]{Go}).

\begin{condition}\label{NN}

$(1)$~$\Hom(I,x)$ is a monomorphism for any injective object $I$ and any inflation $x$.

$(2)$~$\Hom(y,P)$ is a monomorphism for any projective object $P$ and any deflation $y$.
\end{condition}
If $\mathscr{C}$ is a triangulated category or exact category, then Condition \ref{NN} is trivially satisfied.

By \cite[Proposition 5.30]{Go}, for any objects $X,Y\in\mathscr{C}$, we have that $\dim_{k}\mathbb{E}_{\rm{\RNum{1}}}^{i}(X,Y)=\dim_{k}\mathbb{E}_{\rm{\RNum{2}}}^{i}(X,Y)$ for any for $i<0$. Hence, $\{X,Y\}_{\rm{\RNum{1}}}=\{X,Y\}_{\rm{\RNum{2}}}$, and we denote them by $\{X,Y\}$.

\begin{definition}\label{H1} The Hall algebra $\overline{\mathcal{H}}(\mathscr{C})$ of the extriangulated category $\mathscr{C}$ is a $\mathbb{Q}$-space with the basis $\{u_{[X]}~|~X\in \mathscr{C}\}$ and the multiplication defined by
\begin{equation*}u_{[X]}\diamond u_{[Y]}=\sum_{[L]}G_{XY}^{L}u_{[L]}\end{equation*}
where $$G_{XY}^{L}:=\frac{|_{X}(L,Y)|}{|\Aut Y|}\frac{\{L,Y\}}{\{Y,Y\}}=\frac{|(X,L)_{Y}|}{|\Aut X|}\frac{\{X,L\}}{\{X,X\}}.$$
\end{definition}

\begin{remark}\label{We} If $\mathscr{C}$ is a triangulated category, by Corollary \ref{12}, Definition \ref{H1} coincides with that in \cite{XF}. If $\mathscr{C}$ is an exact category, by Corollary \ref{Z}, Definition \ref{H1} coincides with that in \cite{Hu}.

\end{remark}
\begin{theorem}\label{main6} The Hall algebra $\overline{\mathcal{H}}(\mathscr{C})$ of the extriangulated category $\mathscr{C}$ is an associative algebra.
\end{theorem}
\begin{proof}It is proved by the analogous arguments as those for proving Theorem \ref{main3}.
\end{proof}

\section{Comparisons of two Hall algebras}
In this section, we assume that $\mathscr{C}$ satisfies all the conditions in Sections 3 and 6. In particular, $\mathscr{C}$ is {\em locally homologically finite}, i.e.,
for any objects $X,Y\in\mathscr{C}$,
$\sum_{i\in\mathbb{Z}}e^{i}(X,Y)<\infty$, where $e^{i}(X,Y):=\dim_{k}\mathbb{E}^{i}_{\rm{\RNum{1}}}(X,Y)=\dim_{k}\mathbb{E}^{i}_{\rm{\RNum{2}}}(X,Y)$ for each $i\in\mathbb{Z}$.

The {\em Grothendieck group} of the extriangulated category $\mathscr{C}$, denoted by $K(\mathscr{C})$, has been introduced in \cite{ZhuZhuang}, which is similar to those of triangulated categories and exact categories.

For any objects $X,Y\in\mathscr{C}$, define $\lr{X,Y}:=\sum\limits_{i\in\mathbb{Z}}(-1)^ie^{i}(X,Y)$. This sum is finite, since $\mathscr{C}$ is locally homologically finite. Applying Proposition \ref{exact}, we easily obtain that $\lr{-,-}$ descends to give a bilinear form
$$\lr{-,-}:K(\mathscr{C})\times K(\mathscr{C})\longrightarrow\mathbb{Z},$$
known as the {\em Euler form}. In the following, for an object $X$ in $\mathscr{C}$, we denote by $\hat{X}$ the image of $X$ in the Grothendieck group $K(\mathscr{C})$.

Let us consider a twisted version of the Hall algebra $\overline{\mathcal{H}}(\mathscr{C})$. Define $\overline{\mathcal{H}}_{\rm tw}(\mathscr{C})$ to be the same vector space as $\overline{\mathcal{H}}(\mathscr{C})$, but with the twisted multiplication defined by
$$u_{[X]}\ast u_{[Y]}=q^{-\lr{X,Y}}u_{[X]}\diamond u_{[Y]}.$$

\begin{proposition}
The Hall algebra ${\mathcal{H}}(\mathscr{C})$ is just the twisted Hall algebra $\overline{\mathcal{H}}_{\rm tw}(\mathscr{C})$.
\end{proposition}
\begin{proof}
By definition, \begin{equation}\frac{\{X,Y\}}{[X,Y]}=q^{\lr{X,Y}}.\end{equation}
Thus, we have that
\begin{equation*}\begin{split}
u_{[X]}\ast u_{[Y]}&=q^{-\lr{X,Y}}u_{[X]}\diamond u_{[Y]}\\
&=\sum\limits_{[L]} q^{-\lr{X,Y}}\frac{|_{X}(L,Y)|}{|\Aut Y|}\frac{\{L,Y\}}{\{Y,Y\}}u_{[L]}\\
&=\sum\limits_{[L]} q^{-\lr{\hat{L}-\hat{Y},\hat{Y}}}\frac{|_{X}(L,Y)|}{|\Aut Y|}\frac{\{L,Y\}}{\{Y,Y\}}u_{[L]}\\
&=\sum\limits_{[L]} \frac{q^{\lr{{Y},{Y}}}}{q^{\lr{L,Y}}}\frac{|_{X}(L,Y)|}{|\Aut Y|}\frac{\{L,Y\}}{\{Y,Y\}}u_{[L]}\\
&=\sum\limits_{[L]} \frac{\{Y,Y\}}{[Y,Y]}\frac{[L,Y]}{\{L,Y\}}\frac{|_{X}(L,Y)|}{|\Aut Y|}\frac{\{L,Y\}}{\{Y,Y\}}u_{[L]}\\
&=\sum\limits_{[L]}\frac{|_{X}(L,Y)|}{|\Aut Y|}\frac{[L,Y]}{[Y,Y]}u_{[L]}\\
&=\sum\limits_{[L]}F_{XY}^Lu_{[L]}.
\end{split}\end{equation*}
Therefore, we complete the proof.
\end{proof}

\section*{Acknowledgments}
The third author often got asked whether one can generalize the definitions of Hall algebras of exact categories and triangulated categories to extriangulated categories by many peer experts, such as Shengyong Pan, Tiwei Zhao, Panyue Zhou, Bin Zhu, and so on. We are grateful to them for encouraging to write this paper. We also would like to thank Bernhard Keller for suggestions and corrections to the earlier version.


\end{document}